\documentclass[12pt]{article}
\RequirePackage[a4paper,margin=2.54cm]{geometry}

\RequirePackage{amstext}
\RequirePackage{amsmath} 
\RequirePackage{amssymb}
\RequirePackage{graphicx}
\RequirePackage{array}
\RequirePackage{epsfig}
\RequirePackage{stmaryrd}
\RequirePackage[update,prepend]{epstopdf}

\RequirePackage{empheq} 

\RequirePackage[font={small},labelfont=bf,format=hang,format=plain,margin=0pt,width=0.8\textwidth]{caption}
\RequirePackage[list=true]{subcaption}
\RequirePackage{cite}

\RequirePackage{enumitem}
\setlist[itemize]{leftmargin=*}
\setlist[itemize,2]{label=$\circ$, leftmargin=*}

\RequirePackage{etoolbox}
\makeatletter
\patchcmd{\@maketitle}{\LARGE}{\LARGE\bf}{}{}
\makeatother

\makeatletter
\def\ulfootnote{\gdef\@thefnmark{}\@footnotetext}
\makeatother

\RequirePackage[dvipsnames]{xcolor}
\RequirePackage[most]{tcolorbox}
\tcbuselibrary{skins,breakable}

\numberwithin{equation}{section}

\RequirePackage{amsthm}

\newtheorem{prop}{Proposition}[section]
\newtheorem{thm}{Theorem}[section]

\makeatletter
\renewenvironment{proof}[1][\proofname]{\par
\pushQED{\qed}%
\normalfont \topsep6\p@\@plus6\p@\relax
\trivlist
\item\relax
{\bfseries
#1\@addpunct{.}}\hspace\labelsep\ignorespaces
}{%
\popQED\endtrivlist\@endpefalse
}
\makeatother

\RequirePackage{comment}
\newcommand{\addressumu}{Department of Mathematics and Mathematical Statistics, Ume{\aa}~University, SE-901\,87 Ume{\aa}, Sweden}

\newcommand{\addressju}{Department of Mechanical Engineering, J\"onk\"oping University, 
SE-551\,11 J\"onk\"oping, Sweden}

\newcommand{\addressucl}{Department of Mathematics, University 
 College London, Gower Street, London WC1E~6BT, UK}

\newcommand{\bfn}{\boldsymbol n}

\newcommand{\bfp}{\boldsymbol p}

\newcommand{\bfx}{\boldsymbol x}
\newcommand{\bfy}{\boldsymbol y}

\newcommand{\mcK}{\mathcal{K}}

\newcommand{\tn}{|\mspace{-1mu}|\mspace{-1mu}|}
\newcommand{\IR}{\mathbb{R}}

\newcommand{\mcKh}{\mcK_h}

\usepackage{a4wide}
\usepackage{amsmath,amsfonts,amssymb}
\usepackage{bm}
\usepackage{color}
\usepackage[english]{babel}
\usepackage{graphicx}
\usepackage{todonotes}
\usepackage{mathptmx}

\numberwithin{equation}{section}

\title{ Dirichlet Boundary 
Value Correction using Lagrange Multipliers}
\date{\today}

\usepackage{authblk}
\author[1]{Erik Burman}
\author[2]{Peter Hansbo}
\author[3]{Mats G. Larson}
\affil[1]{{\footnotesize\it \addressucl}}
\affil[ ]{{\footnotesize\texttt{e.burman@ucl.ac.uk}}}
\affil[2]{{\footnotesize\it \addressju}}
\affil[ ]{{\footnotesize\texttt{peter.hansbo@ju.se}}}
\affil[2]{{\footnotesize\it \addressumu}}
\affil[ ]{{\footnotesize\texttt{mats.larson@umu.se}}}

\date{\today}

\begin{document}

\maketitle

\begin{abstract}
We propose a boundary value correction approach for cases when curved boundaries are approximated by straight lines (planes) and Lagrange multipliers are used to enforce Dirichlet 
boundary conditions. The approach allows for optimal order convergence for polynomial order up to 3. We show the relation to the Taylor series expansion approach used by Bramble, Dupont 
and Tom\'ee \cite{BrDuTh72} in the context of Nitsche's method and, in the case of  
\emph{inf--sup} stable multiplier methods, prove a priori error estimates with explicit 
dependence on the meshsize and distance between the exact and approximate boundary.\end{abstract}

\section{Introduction}
 
In this contribution we develop a modified Lagrange multiplier method based on the 
idea of boundary value correction originally proposed for standard
finite element methods on an approximate domain in \cite{BrDuTh72} and
further developed in \cite{Du74}. More recently boundary value correction have been 
developed for cut and immersed finite element methods \cite{BuHaLa18,BuHaLa18b,BBCL18, MaSc18,MaSc18b}. Using the closest point mapping to the exact boundary, or an approximation thereof, the boundary condition on the exact boundary may be weakly enforced using multipliers 
on the boundary of the approximate domain. Of particular practical importance 
in this context is the fact that we may use a piecewise linear approximation of the boundary, which is very convenient from a computational point of view since the geometric computations are simple in this case and a piecewise linear distance function may be used to construct the discrete domain. 

We prove optimal order a priori error estimates, in the energy and $L^2$ norms, 
in terms of the error in the boundary approximation and the meshsize. The proof 
utilizes the a priori error estimates derived in \cite{BuHaLa18} for the cut boundary value corrected Nitsche method together with a bound, which shows that the solution to the boundary value corrected Lagrange method is close to the corresponding Nitsche solution for which optimal bounds are available. We 
obtain optimal order convergence for polynomial approximation up to order 3 of the solution. 

Note that without boundary correction one typically requires $O(h^{p+1})$ accuracy in the $L^\infty$ norm for the approximation of the domain, which 
leads to significantly more involved computations on the cut elements for 
higher order elements, see \cite{JoLa13}. We present numerical results 
illustrating our theoretical findings. 

The outline of the paper is as follows: In Section 2 
we formulate the model problem and our method, in Section 3 we present 
our theoretical analysis, in Section 4 we discuss the choice of finite 
element spaces in cut finite element methods, in Section 5 we 
present the numerical results, and finally in Section 6 we include some 
concluding remarks.

\section{Model problem and method}

\subsection{The domain}

Let $\Omega$ be a domain in $\mathbb{R}^d$ with smooth boundary 
$\partial \Omega$ and exterior unit normal $\bfn$. We let $\varrho$ 
be the signed distance function, negative on the inside and 
positive on the outside, to $\partial \Omega$ and we let 
$U_\delta(\partial \Omega)$ 
be the tubular neighborhood $\{\bfx\in \IR^d : |\varrho(\bfx)| < \delta\}$ 
of $\partial \Omega$. Then there is a constant $\delta_0>0$ such 
that the closest point mapping $\bfp(\bfx):U_{\delta_0}(\partial \Omega) 
\rightarrow \partial \Omega$ is well defined and we have the 
identity $\bfp(\bfx) = \bfx - \varrho(\bfx)\bfn(\bfp(\bfx))$. We assume that 
$\delta_0$ is chosen small enough that $\bfp(\bfx)$ is well defined. 
See \cite{GilTru01}, Section 14.6 for further details on distance functions.

\subsection{The model problem}
We consider the problem: find $u:\Omega \rightarrow \IR$ 
such that
\begin{alignat}{2}\label{eq:poissoninterior_strong}
-\Delta u &= f \qquad 
&& \text{in $\Omega$}
\\ \label{eq:poissonbc_strong}
u &= g \qquad && \text{on $\partial\Omega$}
\end{alignat}
where $f\in H^{-1}(\Omega)$ and $g\in H^{1/2}(\partial \Omega)$ 
are given data. It follows from the Lax-Milgram Lemma that there 
exists a unique solution to this problem and we also have the 
elliptic regularity estimate
\begin{equation}\label{eq:ellipticregularity}
\|u\|_{H^{s+2}(\Omega)} \lesssim \|f\|_{H^s(\Omega)}, \qquad 
s \geq -1.
\end{equation} 
Here and below we use the notation $\lesssim$ to denote less or 
equal up to a constant.

Using a Lagrange multiplier to enforce the boundary condition we can write the weak form of \eqref{eq:poissoninterior}--\eqref{eq:poissonbc} as: find
$(u,\lambda) \in H^1(\Omega) \times H^{-1/2}(\partial\Omega)$ such that
\begin{alignat}{2}\label{eq:poissoninterior}
\int_{\Omega}\nabla u \cdot\nabla v \,\text{d}\Omega +\int_{\partial\Omega}\lambda\, v\, \text{d}s  &= \int_{\Omega}f v\,  \text{d}\Omega\qquad \forall v\in H^1(\Omega)
\\ \label{eq:poissonbc}
\int_{\partial\Omega}u\, \mu\, \text{d}s  &= \int_{\partial\Omega}g\, \mu\, \text{d}s\qquad \forall \mu\in H^{-1/2}(\partial\Omega)
\end{alignat}

\subsection{The mesh and the discrete domain}

Let $\mcK_{h}, h \in (0,h_0]$, 
be a family of quasiuniform partitions, with 
mesh parameter $h$, of $\Omega$ into shape 
regular triangles or tetrahedra $K$. 
The partitions induce discrete polygonal approximations $\Omega_h = \cup_{K \in \mcK_h}K$, 
$h \in (0,h_0]$, of $\Omega$. We assume neither $\Omega_h \subset \Omega$ 
nor  $\Omega \subset
\Omega_h$, instead the accuracy with which $\Omega_h$ approximates
$\Omega$ will be crucial. To each $
\Omega_h$ is associated a discrete unit normal $\bfn_h$ and a discrete signed distance 
$\varrho_h:\partial \Omega_h \rightarrow \mathbb{R}$, such that if $\bfp_h(\bfx,\varsigma):=\bfx + \varsigma \bfn_h(\bfx)$ then
$\bfp_h(\bfx,\varrho_h(\bfx)) \in \partial \Omega$ for all $\bfx \in \partial \Omega_h$. We will also assume that $\bfp_h(\bfx,\varsigma)
\in U_{\delta_0}(\Omega):=U_{\delta_0}(\partial\Omega)\cup\Omega$ for all $\bfx \in \partial \Omega_h$ and all
$\varsigma$ between $0$ and $\varrho_h(\bfx)$. For conciseness we will 
drop the second argument of $\bfp_h$ below whenever it takes the value $\varrho_h(\bfx)$, and thus we have the map $\partial \Omega_h \ni \bfx 
\mapsto \bfp_h(\bfx) \in \partial \Omega$. We assume that the following assumptions are satisfied
\begin{equation}\label{eq:geomassum-a}
\delta_h := \| \varrho_h \|_{L^\infty(\partial \Omega_h)} = o(h), 
\qquad  h \in (0,h_0]
\end{equation}
and 
\begin{equation}\label{eq:geomassum-c}
\| \bfn_h - \bfn\circ \bfp_h \|_{L^\infty(\partial \Omega_h)} = o(1), 
\qquad  h \in (0,h_0]
\end{equation}
where $o(\cdot)$ denotes the little ordo. We also assume 
that $h_0$ is small enough to guarantee that 
\begin{equation}\label{eq:geomassum-b}
\partial \Omega_h \subset U_{\delta_0}(\partial \Omega), \qquad h\in(0,h_0]
\end{equation}
and that  there exists $M>0$ such for any $\bfy \in U_{\delta_0}(\partial
\Omega)$ the equation, find $\bfx \in \partial \Omega_h$ and $
|\varsigma| \leq \delta_h$ such that
\begin{equation}\label{eq:assump_olap}
{\bfp}_h(\bfx,\varsigma) = \bfy 
\end{equation}
has a solution set $\mathcal{P}_h$ with
\begin{equation}\label{eq:card_hyp}
\mbox{card}(\mathcal{P}_h) \leq M
\end{equation}
uniformly in $h$. The rationale of this assumption is to ensure that
the image of $\bfp_h$ can not degenerate for
vanishing $h$; for more information cf. \cite{BuHaLa18}.

 We note that it follows from (\ref{eq:geomassum-a}) that 
\begin{equation}\label{eq:geomassum-exact-normal}
\|\varrho \|_{L^\infty(\partial \Omega_h)} 
\lesssim 
\|\varrho_h \|_{L^\infty(\partial \Omega_h)} 
= o(h)
\end{equation}
since $|\varrho_h(\bfx)| \geq |\varrho(\bfx)|$, $\bfx \in U_{\delta_0}(\partial \Omega)$.
We also assume the additional regularity
\begin{equation}\label{eq:residualregularity}
f+ \Delta u \in H^{l\textcolor{black}{+\frac12+\epsilon}}(U_{\delta_0}(\Omega))
\end{equation}

\subsection{The finite element method} 
\subsubsection{Boundary value correction} 

The basic idea of the boundary value correction of \cite{BrDuTh72} is to use a Taylor series at $\bfx \in {\partial\Omega_h}$ in the direction $\bfn_h$, and let 
this series represent $u_h \vert_{\partial\Omega}$. In the present work we will restrict ourselves to 
\begin{equation}\label{def:Taylor}
u_h\circ \bfp_h(\bfx) \approx u_h(\bfx) + \varrho_h(\bfx)\bfn_h(\bfx)\cdot\nabla u_h(\bfx)
\end{equation}
which is the case of most practical interest.

Choosing appropriate discrete spaces $V_h$ and $\Lambda_h$ for the approximation of $u$ and $\lambda$, respectively (particular choices are considered in Section \ref{sec:numex}),
we thus seek $(u_h,\lambda_h)\in V_h\times\Lambda_h$ such that
\begin{alignat}{2}\label{eq:poissoninterior_FEM}
\int_{\Omega_h}\nabla u_h \cdot\nabla v \,\text{d}\Omega_h +\int_{\partial\Omega_h}\lambda_h\, v\, \text{d}s  &= \int_{\Omega_h}f v\,  \text{d}\Omega_h\qquad \forall v\in V_h
\\ \label{eq:poissonbc_FEM}
\int_{\partial\Omega_h}(u_h+\varrho_h\bfn_h\cdot\nabla u_h)\, \mu\, \text{d}s  &= \int_{\partial\Omega_h}\tilde{g}\, \mu\, \text{d}s\qquad \forall \mu\in \Lambda_h
\end{alignat}
where we introduced the notation $\tilde{g}:= g\circ \bfp_h$ for the pullback of
 $g$ from $\partial \Omega$ to $\partial \Omega_h$.

Using Green's formula we note that the first equation implies that $\lambda_h = -\bfn_h\cdot\nabla u_h$, and therefore we now propose the following modified method: Find $(u_h,\lambda_h)\in V_h\times\Lambda_h$ such that
\begin{alignat}{2}\label{eq:multinterior}
\int_{\Omega_h}\nabla u_h \cdot\nabla v \,\text{d}\Omega_h +\int_{\partial\Omega_h}\lambda_h\, v\, \text{d}s  &= \int_{\Omega_h}f v\,  \text{d}\Omega_h\qquad \forall v\in V_h
\\ \label{eq:multbc}
\int_{\partial\Omega_h}u_h\, \mu\, \text{d}s-\int_{\partial\Omega_h} \varrho_h\lambda_h\,\mu\, \text{d}s  &= \int_{\partial\Omega}\tilde{g}\, \mu\, \text{d}s\qquad \forall \mu\in \Lambda_h
\end{alignat}
or
\begin{equation}
A(u_h,\lambda_h;v,\mu) = (f,v)_{\Omega_h} + (\tilde{g},\mu)_{\partial\Omega_h}\quad \forall (u_h,\lambda_h)\in V_h\times\Lambda_h\label{eq:mainproblem}
\end{equation}
where $(\cdot,\cdot)_{M}$ denotes the $L_2$ scalar product over $M$, with $\| \cdot\|_{M}$ the corresponding $L_2$ norm,
and
\[
A(u,\lambda;v,\mu) := (\nabla u ,\nabla v )_{\Omega_h} +(\lambda , v)_{\partial\Omega_h} +( u,\mu)_{\partial\Omega_h} -(\varrho_h\lambda ,\mu)_{\partial\Omega_h}.
\]

%

\subsection{Relation to Nitsche's method with boundary value correction}
Problem (\ref{eq:mainproblem}) can equivalently be formulated as finding the stationary points of the Lagrangian
\begin{equation}
\mathcal{L}(u,\lambda) := \frac12\|\nabla u\|^2_{\Omega_h} +  (\lambda,u)_{\partial\Omega_h}-\|\varrho^{1/2}_h\lambda\|^2_{\partial\Omega_h}
-(f,u)_{\Omega_h} - (\tilde{g},\lambda)_{\partial\Omega_h}
\end{equation}
We now follow \cite{BuHa17} and add a consistent penalty term and seek stationary points of the augmented Lagrangian
\begin{equation}
\mathcal{L}_\text{aug}(u,\lambda) := 
\mathcal{L}(u,\lambda) + \frac{1}{2} \|\gamma^{1/2}(u-\varrho_h\lambda-\tilde{g})\|^2_{\partial\Omega_h}
\end{equation}
where $\gamma > 0$ remains to be chosen. The corresponding optimality system is 
\begin{align}\nonumber
(f,v)_{\Omega_h} + (\tilde{g},\mu)_{\partial\Omega_h} = {}& A(u_h,\lambda_h;v,\mu) 
 +(\gamma (u_h-\varrho_h\lambda_h-\tilde{g}),v)_{\partial\Omega_h}\\
& -(\gamma\varrho_h (u_h-\varrho_h\lambda_h-\tilde{g}),\mu)_{\partial\Omega_h} 
\end{align}
Now, formally replacing $\lambda_h$ by $-\bfn_h\cdot\nabla u_h$ and $\mu$ by $-\bfn_h\cdot\nabla v$ we obtain
\begin{align}\nonumber
(f,v)_{\Omega_h} - (\tilde{g},\bfn_h\cdot\nabla v)_{\partial\Omega_h} = {}& (\nabla u_h ,\nabla v )_{\Omega_h} -(\bfn_h\cdot\nabla u_h,v)_{\partial\Omega_h}  \\ \nonumber
 & -(u_h,\bfn_h\cdot\nabla v)_{\partial\Omega_h}
-(\varrho_h \bfn_h\cdot\nabla u_h,\bfn_h\cdot\nabla v)_{\partial\Omega_h}\\
&  +(\gamma (u_h+\varrho_h\bfn_h\cdot\nabla u_h-\tilde{g}),v+\varrho_h\bfn_h\cdot\nabla v)_{\partial\Omega_h}
\end{align}
Setting now $\gamma = \gamma_0/h$, with $\gamma_0$ sufficiently large
to ensure coercivity, we obtain the symmetrized version of the
boundary value corrected Nitsche method proposed in \cite{BrDuTh72}
with optimal convergence up to order $p=3$ assuming $\varrho_h\ge - C
h$, for some sufficiently small constant. This means that
$\partial \Omega_h$ either has to be a good approximation of $\partial
\Omega$, or where
it approximates poorly, $\Omega_h$ must approximation $\Omega$ from
the inside. For future reference we
define this method as: Find $u_h \in V_h$ such that
\begin{equation}\label{eq:Nitform}
A_{Nit}(u_h,v_h) = (f,v_h)_{\partial \Omega_h} + (\tilde g, \bfn_h
\cdot \nabla v_h)_{\partial \Omega_h}+ (\gamma \tilde g, v_h + \varrho_h \bfn_h
\cdot \nabla v_h)_{\partial \Omega_h}
\end{equation}
for all $v_h \in V_h$.
Here the bilinear form is defined by
\begin{align}\nonumber
A_{Nit}(w_h,v_h) 
&:= (\nabla w_h ,\nabla v_h )_{\Omega_h}
-(\bfn_h\cdot\nabla w_h,v_h+\varrho_h \bfn_h\cdot\nabla v_h)_{\partial\Omega_h}  -(w_h+\varrho_h \bfn_h\cdot\nabla
w_h,\bfn_h\cdot\nabla
v)_{\partial\Omega_h} 
\\
&\qquad +(\varrho_h \bfn_h\cdot\nabla
w_h,\bfn_h\cdot\nabla v)_{\partial\Omega_h}
+ (\gamma (w_h+\varrho_h\bfn_h\cdot\nabla w_h,v+\varrho_h\bfn_h\cdot\nabla v_h)_{\partial\Omega_h}.
\end{align}

\section{Elements of analysis}
In this section we will prove some basic results on the stability and
the accuracy of the method (\ref{eq:mainproblem}). We will restrict
ourselves to a discussion of the case $- C h \leq \varrho_h$, for some
$C$ small enough. We assume that $\Lambda_h$ is the space of piecewise
polynomial functions of order $k-1$ and $V_h$ is the space of
continuous piecewise polynomial functions of order $k$, that we will
denote $V_h^k$, enriched with
higher order bubbles on the faces in $\partial \Omega_h$ so that
inf-sup stability holds. The precise condition is given in equation
(\ref{eq:infsup}) below. For details on stable choices of the
multiplier space we refer to \cite{BM97, BD98, KLPV01}. We introduce the triple norm defined on $
H^1(\Omega_h) \times L^2(\partial \Omega_h)$:
\begin{equation}
\tn (v,\mu) \tn := \|\nabla v\|_{\Omega_h} + \|h^{-\frac12}
v\|_{\partial \Omega_h} + \|h^{\frac12} \mu\|_{\partial \Omega_h}.
\end{equation}
We let $\pi_h:L^2(\partial \Omega_h) \to \Lambda_h$ denote the
$L^2$-orthogonal projection and we assume that the bound 
\begin{equation}
\|v - \pi_h v\|_{\partial \Omega_h} \lesssim h \|\nabla_\partial v\|_{\partial \Omega_h}
\end{equation}
for all $v \in H^1(\partial \Omega_h)$ and where $\nabla_\partial$
denotes the gradient on the boundary. 
The formulation (\ref{eq:mainproblem}) satisfies the following
stability result
\begin{prop}\label{prop:infsup}
Assume that $\varrho_h \ge - C_{\partial \Omega} h$ and that $V_h
\times \Lambda_h$ satisfies the inf-sup condition. Then for $C_{\partial
  \Omega}$ sufficiently small, for all $(y_h,\eta_h) \in V_h \times \Lambda_h$, 
  there exists
$(v_h,\mu_h) \in V_h \times \Lambda_h$ such that
\begin{equation}
\tn (y_h,\eta_h) \tn \lesssim A(y_h,\eta_h;v_h,\mu_h)
\end{equation}
and 
\begin{equation}
\tn (v_h,\mu_h) \tn \lesssim \tn  (y_h,\eta_h)\tn.
\end{equation}
\end{prop}
\begin{proof}
First observe that
\begin{equation}
A(y_h,\eta_h;y_h,-\eta_h) = \|\nabla y_h\|^2_{\Omega_h} + (\varrho_h
\eta_h, \eta_h)_{\partial \Omega_h}.
\end{equation}
Then recall that since the space satisfies the inf-sup condition there
exists $v_\eta \in V_h$ such that
\begin{equation}\label{eq:infsup}
\|h^{\frac12} \eta_h\|_{\partial \Omega_h} \lesssim
(\eta_h,v_\eta)_{\partial \Omega_h} \quad \mbox{and} \quad  \|\nabla
v_\eta\|_{\Omega_h}+\|h^{-\frac12}
v_\eta\|_{\partial \Omega_h} \lesssim \|h^{\frac12} \eta_h\|_{\partial \Omega_h}.
\end{equation}
It follows that for some $c_\eta,\, C_{\partial \Omega_h}$
sufficiently small 
\begin{align}
\|\nabla y_h\|^2_{\Omega_h}+\|h^{\frac12} \eta_h\|^2_{\partial \Omega}
&\lesssim \|\nabla y_h\|^2_{\Omega_h} + (\varrho_h
\eta_h, \eta_h)_{\partial \Omega_h} + \|h^{\frac12}
\eta_h\|^2_{\partial \Omega} 
\\
&\lesssim A(y_h,\eta_h;y_h + c_\eta v_\eta,-\eta_h ).
\end{align}
Here we used equation \eqref{eq:infsup},
\begin{equation}
(\varrho_h
\eta_h, \eta_h)_{\partial \Omega_h} + \|h^{\frac12}
\eta_h\|^2_{\partial \Omega} \ge (1 - C_{\partial \Omega_h}) \|h^{\frac12}
\eta_h\|^2_{\partial \Omega}
\end{equation}
and 
\begin{equation}
(\nabla y_h, y_h + c_\eta v_\eta)_{\Omega_h} \ge -\frac12 \|\nabla
y_h\|^2_{\Omega_h} - 2 c_\eta^2 \|v_\eta\|^2_{\partial \Omega}.
\end{equation}
Finally let $\mu_y =\pi_h y_h$ and observe that
\begin{align}\nonumber
&(y_h, h^{-1}  \mu_y)_{\partial \Omega} - (\rho_h \eta_h, h^{-1} \mu_y)
_{\partial \Omega} 
\\
 &\qquad \ge \|h^{-\frac12}
y_h\|_{\partial \Omega_h}^2 - \|h^{-\frac12} (y_h - \mu_y)\|_{\partial
  \Omega_h}^2 -  \frac12  C_{\partial \Omega_h}^2 \|h^{\frac12} \eta_h\|_{\partial
  \Omega_h}^2 - \frac12 \|h^{-\frac12} \mu_y\|_{\partial
  \Omega_h}^2 
  \\
  &\qquad 
\ge \frac12 \|h^{-\frac12}
y_h\|_{\partial \Omega_h}^2 - C^2_t \|\nabla y_h\|^2_{\Omega_h} -  \frac12  C_{\partial \Omega}^2 \|h^{\frac12} \eta_h\|_{\partial
  \Omega_h}^2.
\end{align}
Where we used the approximation property of $\pi_h$ and a trace
inequality 
\begin{equation}\label{eq:trace}
h_K^{\frac12}\|v_h\|_{\partial K} + h_K \|\nabla v_h\|_{K} \lesssim \|v_h\|_K.
\end{equation}
for all elements  $K \in \mcKh$ and polynomials $v_h \in \mathbb{P}(K)$, 
to show that
\begin{equation}
\|h^{-\frac12} (y_h - \mu_y)\|_{\partial
  \Omega_h} \leq C_t \|\nabla y_h\|_{\Omega_h}.
\end{equation}
The first claim follows by taking $v_h = y_h + c_ \eta v_\eta$ and
$\mu_h = - \eta_h + c_y h^{-1} \mu_y$ with $c_\eta$ and $c_y$ both $O(1)$,
sufficiently small and assuming that $C_{\partial \Omega_h}$ is small
enough.

To conclude the proof we need to show that
\begin{equation}
\tn (v_h,\mu_h) \tn \lesssim \tn  (y_h,\eta_h) \tn.
\end{equation}
By the triangle inequality we have
\begin{equation}
\tn (v_h,\mu_h) \tn \leq \tn  (y_h,\eta_h) \tn+\tn  ( c_ \eta v_\eta,c_y h^{-1} \mu_y) \tn.
\end{equation}
By definition
\begin{equation}
\tn  ( c_ \eta v_\eta,c_y h^{-1} \mu_y) \tn =c_ \eta \|\nabla v_\eta\|_{\Omega_h} + c_ \eta\|h^{-\frac12}
v_\eta\|_{\partial \Omega_h} + c_y \|h^{-\frac12} \mu_y\|_{\partial \Omega_h}
\end{equation}
and the proof follows from (\ref{eq:infsup}) together with the stability of $\pi_h$ 
in $L^2$.
\end{proof}
We will now use this stability result to prove an error estimate. For
simplicity we here assume that $\varrho_h>0$, i.e. $\Omega_h \subset \Omega$.
\begin{thm}
Let $u \in H^{k+1}(\Omega)$ denote the solution to (\ref{eq:poissoninterior})--(\ref{eq:poissonbc}).
Let $u_h, \lambda_h \in V_h \times \Lambda_h$ denote the solution of
(\ref{eq:mainproblem}). Assume that the polynomial order of $V_h$ is
$k \in \{1,2,3\}$, with enrichment on the boundary and $\Lambda_h \equiv X_h^{k-1}$. Assume that
$V_h \times \Lambda_h$ satisfies (\ref{eq:infsup}).
Then there holds, with $\tilde \lambda = \bfn_h \cdot \nabla
u\vert_{\partial \Omega_h}$,
\begin{align}\label{eq:errorestenergy}
\tn (u - u_h, \tilde \lambda- \lambda_h) \tn
&\lesssim 
h^{k} \|u \|_{H^{k+1}(\Omega)}
+ 
h^{-1/2} \delta_h^{2} 
\sup_{0\leq t \leq \delta_0} \| D^{k+1} u\|_{L^2(\partial \Omega_t)}
\\ \nonumber
&\qquad + 
\textcolor{black}{h^{1/2}} \delta_h^{l+1} 
\sup_{-\delta_0\leq t < 0} 
\| D_n^{l} (f + \Delta u)\|_{L^2(\partial \Omega_t)}.
\end{align}
\end{thm}
\begin{proof}
Let $\tilde u_h \in V_h$ denote the solution to (\ref{eq:Nitform}). We
recall from \cite{BrDuTh72, BuHaLa18} that the following error bound
holds 
\begin{align}\nonumber
\tn (u - \tilde u_h,0) \tn&+\|h^{\frac12} \bfn_h \cdot \nabla (u -
  \tilde u_h)\|_{\partial \Omega_h}+ \|h^{-\frac12} (\tilde u_h +
  \varrho \bfn_h \cdot \nabla \tilde u_h - \tilde g)\|_{\partial
                            \Omega_h} \\ \label{eq:errorNit}
&\lesssim 
h^{k} \|u \|_{H^{k+1}(\Omega)}
+ 
h^{-1/2} \varrho_h^{2} 
\sup_{0\leq t \leq \delta_0} \| D^{k+1} u\|_{L^2(\partial \Omega_t)}
\\ \nonumber
&\qquad + 
\textcolor{black}{h^{1/2}} \varrho_h^{l+1} 
\sup_{-\delta_0\leq t < 0} 
\| D_n^{l} (f + \Delta u)\|_{L^2(\partial \Omega_t)}.
\end{align}
Let $i_h u$ denote the nodal interpolant of $u$. We then form the
discrete errors $e_h =  u_h - \tilde u_h$ and $\varsigma_h = \lambda_h
- \zeta_h$ for some $\zeta_h \in \Lambda_h$. Using the triangle inequality and we have
\begin{equation}
\tn (u - u_h, \tilde \lambda- \lambda_h) \tn \leq \tn (u - \tilde u_h,
\tilde \lambda- \zeta_h) \tn + \tn (e_h, \varsigma_h) \tn.
\end{equation}
Since the first term on the left hand side is bounded by standard
interpolation and (\ref{eq:errorNit}). We only need to consider the
second term. By the stability estimate of Proposition
\ref{prop:infsup} we have
\begin{equation}
\tn (e_h, \varsigma_h) \tn \lesssim A(e_h,\varsigma_h; v_h,\mu_h).
\end{equation}
Using the method (\ref{eq:mainproblem}) and the definition of $\tilde u_h$ we 
find that
\begin{align} \label{eq:gal_ortho}
A(e_h,\varsigma_h; v_h,\mu_h) & = (f,v_h)_{\Omega_h} + (\tilde g,
  \mu_h)_{\partial \Omega_h} 
  \\ \nonumber
&\qquad - (\nabla \tilde u_h,\nabla v_h) _{\Omega_h}+(\zeta_h,v_h)_{\partial \Omega_h}
\\ \nonumber
&\qquad -(\tilde u_h, \mu_h)_{\partial \Omega_h} - (\varrho_h\zeta_h, \mu_h) _{\partial \Omega_h}.
\end{align}
The definition of Nitsche's method (\ref{eq:Nitform}) implies the equality
\begin{align}
(f,v_h)_{\Omega_h} -  (\nabla \tilde u_h ,\nabla v_h )_{\Omega_h} 
= {}& (\tilde{g},\bfn_h\cdot\nabla v_h)_{\partial\Omega_h} -(\bfn_h\cdot\nabla \tilde u_h,v_h)_{\partial\Omega_h}  
\\ \nonumber
 & -(\tilde u_h,\bfn_h\cdot\nabla v_h)_{\partial\Omega_h}
-(\varrho_h \bfn_h\cdot\nabla \tilde u_h,\bfn_h\cdot\nabla v_h)_{\partial\Omega_h}
\\ \nonumber
&  +(\gamma (\tilde u_h+\varrho_h\bfn_h\cdot\nabla \tilde
  u_h-\tilde{g}),v_h+\varrho_h\bfn_h\cdot\nabla
  v_h)_{\partial\Omega_h}
\\ 
= {} &  -(\bfn_h\cdot\nabla \tilde u_h,v_h)_{\partial\Omega_h} + (\tilde{g} - \tilde u_h - \varrho_h \bfn_h\cdot\nabla \tilde
       u_h,\bfn_h\cdot\nabla v_h)_{\partial\Omega_h}  
\\ \nonumber
& +(\gamma (\tilde u_h+\varrho_h\bfn_h\cdot\nabla \tilde
  u_h-\tilde{g}),v_h+\varrho_h\bfn_h\cdot\nabla v_h)_{\partial\Omega_h}.
\end{align}
Combining  then (\ref{eq:gal_ortho}) with (\ref{eq:Nitform}) we have
\begin{align}
A(e_h,\varsigma_h; v_h,\mu_h) = &(\tilde g-\tilde u_h-\varrho_h \zeta_h,
  \mu_h)_{\partial \Omega_h} \\ \nonumber
&+(\zeta_h-\bfn_h \cdot
  \nabla \tilde u_h,v_h)_{\partial \Omega_h}\\ \nonumber
& + (\tilde{g} - \tilde u_h - \varrho_h \bfn_h\cdot\nabla \tilde
       u_h,\bfn_h\cdot\nabla v_h)_{\partial\Omega_h}  \\ \nonumber
& +(\gamma (\tilde u_h+\varrho_h\bfn_h\cdot\nabla \tilde
  u_h-\tilde{g}),v_h+\varrho_h\bfn_h\cdot\nabla v_h)_{\partial \Omega_h}\\
= & I+II+III+IV.
\end{align}
We will now bound the terms $I-IV$. 

First note that,
\begin{align}
I+III+IV \leq & (\|h^{-\frac12}(\tilde g-\tilde u_h-\varrho_h
\bfn_h \nabla\cdot \tilde u_h)\|_{\partial \Omega_h}+\|h^{-\frac12}\varrho_h (\zeta_h-
\bfn_h \cdot \nabla \tilde u_h) \|_{\partial \Omega_h})
\\ \nonumber
& \times 
 (\| \nabla
v_h\|_{\Omega_h} +  \|h^{-\frac12}
  v_h\|_{\partial \Omega_h} + \|h^{\frac12} \mu_h\|_{\partial \Omega_h}).
\end{align}
For term $II$ there holds using Cauchy-Schwarz inequality
\begin{align}
II=(\zeta_h-\bfn_h \cdot
  \nabla \tilde u_h,v_h)_{\partial \Omega_h} \lesssim 
  \|h^{\frac12}(\zeta_h - \bfn_h \cdot\nabla \tilde u_h)\|_{\partial \Omega_h} \|h^{-\frac12}
  v_h\|_{\partial \Omega_h}.
\end{align}
Summing up we have using the assumption that
$\|\rho_h\|_{L^\infty(\partial \Omega_h)} \leq O(h)$,
\begin{align}
I+II+III+IV &\leq (\|h^{-\frac12}(\tilde g-\tilde u_h-\varrho_h
\bfn_h \cdot \nabla \tilde u_h)\|_{\partial \Omega_h}
\\ \nonumber
&\qquad +\|h^{\frac12} (\zeta_h-
\bfn_h \cdot \nabla \tilde u_h) \|_{\partial \Omega_h}) \tn (v_h,\mu_h)\tn.
\end{align}

For the term $\|h^{-\frac12}(\tilde g-\tilde u_h-\varrho_h \bfn_h
  \cdot \nabla \tilde u_h)\|_{\partial \Omega_h}$ we may use the bound
  (\ref{eq:errorNit}). It only remains to bound $\|h^{\frac12} (\zeta_h-
\bfn_h \nabla \tilde u_h) \|_{\partial \Omega_h} $. To this end consider
$\pi_{k-1} \nabla \tilde u_h \in [X_h]^d$ and let $\zeta_h = \bfn_h
\cdot \pi_{k-1} \nabla \tilde u_h\vert_{\partial \Omega_h}$. For this
choice we have using a trace inequality
\begin{align}
\|h^{\frac12} (\zeta_h-
\bfn_h \nabla \tilde u_h) \|_{\partial \Omega_h}  = {}&\|h^{\frac12} \bfn_h
\cdot (\pi_{k-1} \nabla \tilde u_h- \nabla \tilde u_h) \|_{\partial
  \Omega_h} 
  \\
 \leq & \|\pi_{k-1} \nabla \tilde u_h- \nabla \tilde u_h\|_{\Omega_h}.
\end{align}
To bound the term in the right hand side we add and subtract
$\nabla u - \pi_{k-1} \nabla u$ and use the triangle inequality and
the stability of the $L^2$-projection $\pi_{k-1}$ to
obtain
\begin{align}
\|\pi_{k-1} \nabla \tilde u_h- \nabla \tilde u_h\|_{\Omega_h} &\leq
\|\pi_{k-1} (\nabla \tilde u_h- \nabla u)\|_{\Omega_h}+
  \|\pi_{k-1} \nabla u- \nabla u\|_{\Omega_h}+\|\nabla u- \nabla
  \tilde u_h\|_{\Omega_h} 
  \\
& \leq \|\pi_{k-1} \nabla u- \nabla u\|_{\Omega_h}+2 \|\nabla u- \nabla
  \tilde u_h\|_{\Omega_h}.
\end{align}
For the first term in the right hand side we have the approximation
bound
\begin{equation}
\|\pi_{k-1} \nabla u- \nabla u\|_{\Omega_h} \lesssim h^k \|D^{k+1} u\|_{\Omega_h}.
\end{equation}
The second term is bounded by (\ref{eq:errorNit}).
We conclude by applying the second inequality of Proposition
\ref{prop:infsup}.
\end{proof}
\section{Remarks on cut finite element methods}
In the context of cut finite element methods the discontinuous
multiplier spaces used above can no longer be expected to be
stable. It is possible to stabilise the multiplier using
Barbosa-Hughes stabilisation. However, fluctuation based multipliers
are unlikely to be suitable in this context since the weak
consistency of the fluctuations of the multiplier between elements
depends on the geometry approximation through the interface normal.
Since the method is of interest when the geometry approximation is of
relatively low order, this limits the possibility to use fluctuation
based stabilisation.

For closed smooth boundaries, one may prove inf-sup stability and optimal
convergence, without stabilisation, when using continuous approximation of polynomial order
less than or equal to $2$, for both the
bulk variable and the multiplier provided $\rho_h = O(h^2)$. The approximation order of the interface normal,
which is $O(h)$ prohibits higher order convergence if the interface
approximation is piecewise affine. For instance, piecewise cubic
continuous approximation will not necessarily achieve higher order
convergence that the piecewise quadratic approximation.

\section{Numerical examples}\label{sec:numex}

We show examples of higher order triangular elements with linearly interpolated boundary and low order rectangular elements with staircase boundary, using discontinuous multiplier spaces.
In all examples we define the meshsize $h=1/\sqrt{\text{NNO}}$, where NNO corresponds to the number of nodes of the lowest order FEM on the mesh in question (bilinear or affine). 

\subsection{Triangular elements}

We first consider the case of affine triangulations of a ring $1/4\leq r\leq 3/4$, $r=\sqrt{x^2+y^2}$. We use the manufactured solution $u=(r-1/4)(3/4-r)$ and compute the corresponding right--hand side analytically. An elevation of the a typical discrete solution is given in Fig.\  \ref{fig:trisol}.

We use continuous piecewise $P^k$ polynomials, $k=2,3$ for the approximation of $u$, and for the approximation of $\lambda$ we use 
piecewise $P^{k-1}$ polynomials, discontinous on each element edge on $\Gamma_h$. To ensure {\em inf--sup}\, stability, we add hierarchical $P^{k+1}$ bubbles
on each edge in the approximation of $u$.

\paragraph{Second order elements.}
In Fig. \ \ref{fig:errtri} we show the convergence in $L_2(\Omega_h)$ and $H^1(\Omega_h)$ with and without boundary modification.
In Fig. \ \ref{fig:errlam} we show the error in multiplier computed as $\| (-\bfn\cdot\nabla u)\vert_{\partial\Omega_h} - \lambda_h\|_{\partial\Omega_h}$. Optimal order convergence is observed for the modified method, convergence $O(h^3)$ in $L_2(\Omega_h)$ and $O(h^2)$ 
in $H^1(\Omega_h)$; the multiplier error is approximately $O(h^2)$.

\paragraph{Third order elements.}
Next we use continuous piecewise third order polynomials for the approximation of $u$, and for the approximation of $\lambda$ we use piecewise quadratic polynomials, discontinous on each element edge on $\Gamma_h$.  In Fig.\ \ref{fig:errtri2} we show the convergence in $L_2(\Omega_h)$ and $H^1(\Omega_h)$ with and without boundary modification.
In Fig.\  \ref{fig:errlam2} we show the error in multiplier computed as above. Optimal order convergence is again observed for the modified method, convergence $O(h^4)$ in $L_2(\Omega_h)$ and $O(h^3)$ in $H^1(\Omega_h)$; the multiplier error is approximately $O(h^3)$. Note that no improvement over $P^2$ approximations can be seen in the unmodified method due to the geometry error being dominant.

\paragraph{An unstable pairing of spaces.}
We finally make the observation that our modification has a stabilising influence on the approximation. We try continuous $P^2$ approximations of $u$ and discontinuous $P^2$ approximations of $\lambda$. In this case we get no convergence without the modification due to the violation of the {\em inf--sup} condition,
whereas with modification we obtain the optimal convergence pattern in $u$ and a stable multiplier convergence given in Fig \ref{fig:conunstab}.

\subsection{Rectangular elements}

This example shows that it is possible to achieve optimal convergence even on a staircase boundary. We use a continuous piecewise $Q_1$ approximation on the (affine) rectangles, again enhanced for {\em inf--sup}, now by hierarchical $P^2$ bubble function on the boundary edges, together with edgewise constant multipliers on $\Gamma_h$. We use the manufactured solution $u=\sin(x^3)\cos(8y^3)$ on the domain inside the ellipse $x^2/4+y^2 = 1$. Our computational grids consist of elements completely inside this ellipse; a typical coarse grid is shown if Fig.\ \ref{fig:coarse} where we note the staircase boundary.
In Fig.\ \ref{fig:elevq} we show elevations of the numerical solutions on a finer grid without and with boundary correction. In Fig.\ \ref{fig:errquad} we show the errors of the unmodified and modified methods. Again we observe optimal order convergence for the modified method, $O(h^2)$ in $L_2(\Omega_h)$ and $O(h)$ in $H^1(\Omega_h)$.

\section{Concluding remarks}

We have introduced a symmetric modification of the Lagrange multiplier approach to satisfying Dirichlet boundary conditions for Poisson's equation.
This novel approach allows for affine approximations of the boundary, and thus affine elements, up to polynomial approximation order 3 without loss of
convergence rate as compared to higher order boundary fitted meshes. The modification is easy to implement and only requires that the distance to the exact boundary in the direction of the discrete normal can be easily computed. In fact, the modification stabilises the multiplier method so that unstable pairs of spaces can be used, as long as there is a uniform distance to the boundary.

\paragraph{Acknowledgement.} This research was supported in part by
EPSRC, UK, Grant No. \newline EP/P01576X/1, the Swedish Foundation for Strategic Research Grant No.\ AM13-0029, the Swedish Research Council Grants No.
2013-4708,  2017-03911, 2018-05262, and Swedish strategic research programme eSSENCE.

\bibliographystyle{abbrv}

\newpage

\begin{figure}[ht]
\begin{center}\includegraphics[scale=0.35]{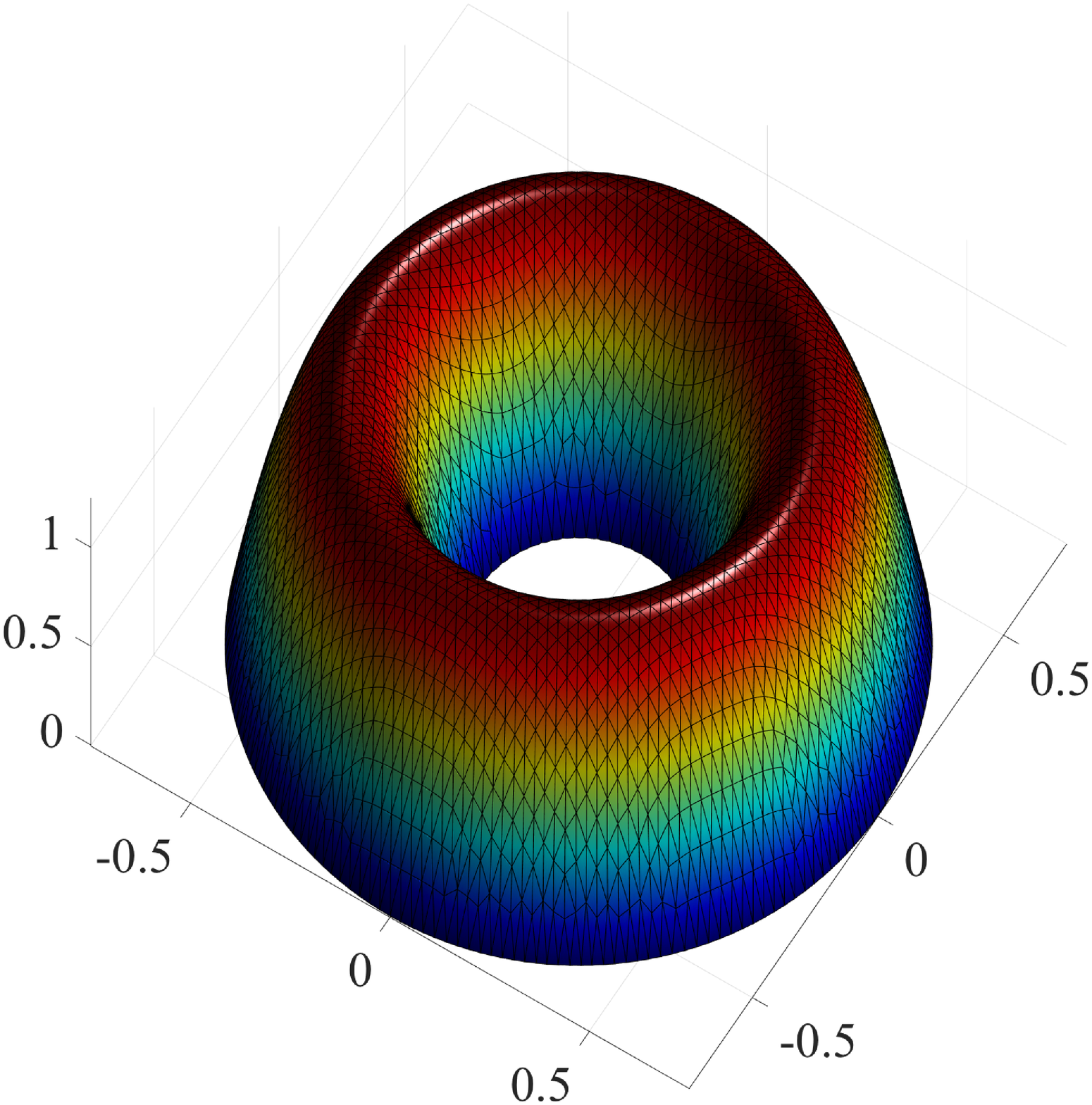}
\end{center}
\caption{Elevation of the discrete solution on triangles.} \label{fig:trisol}
\end{figure}
\begin{figure}[ht]
\begin{center}\includegraphics[scale=0.3]{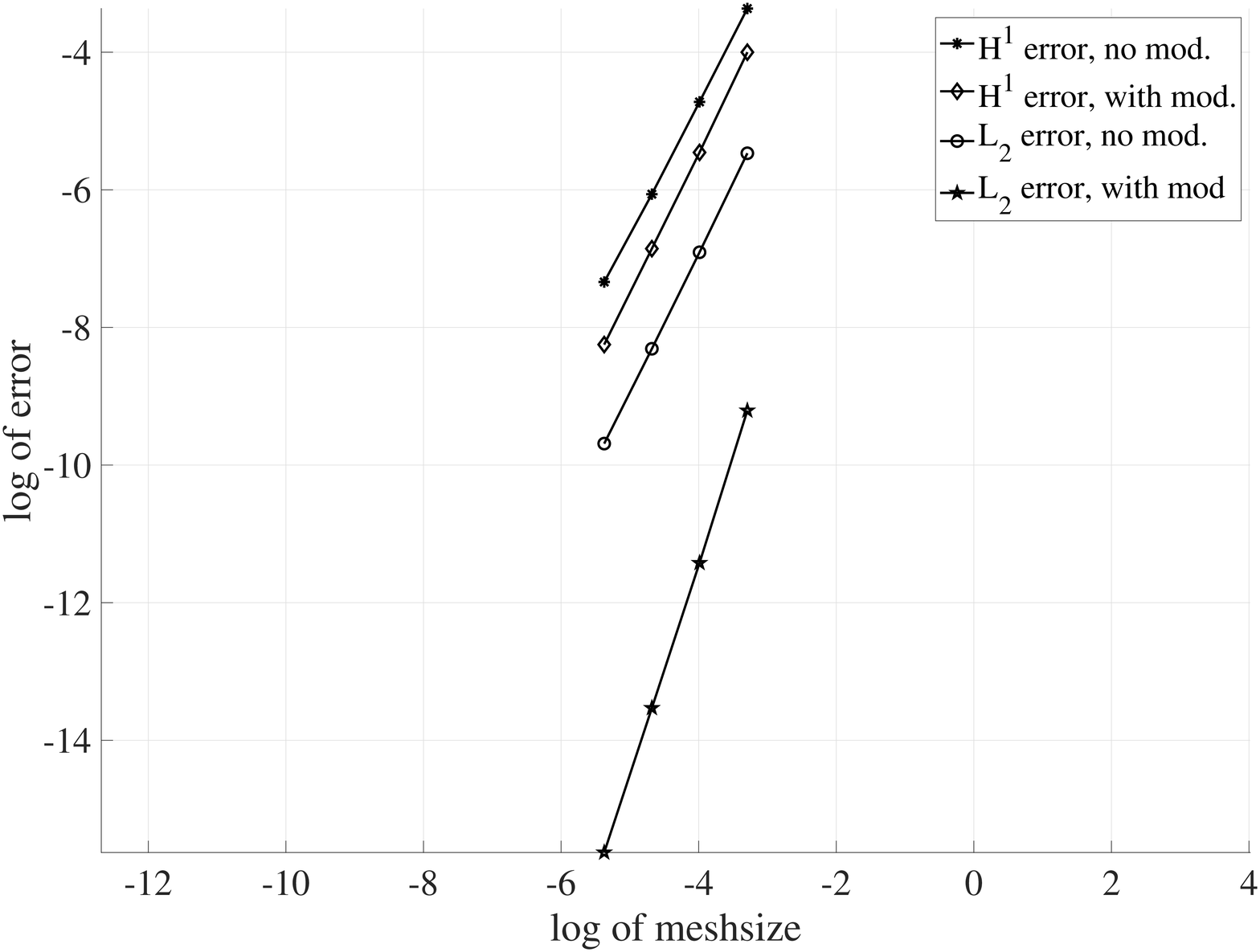}
\end{center}
\caption{Errors with and without boundary modification, $P^2$ case.} \label{fig:errtri}
\end{figure}
\begin{figure}[ht]
\begin{center}\includegraphics[scale=0.3]{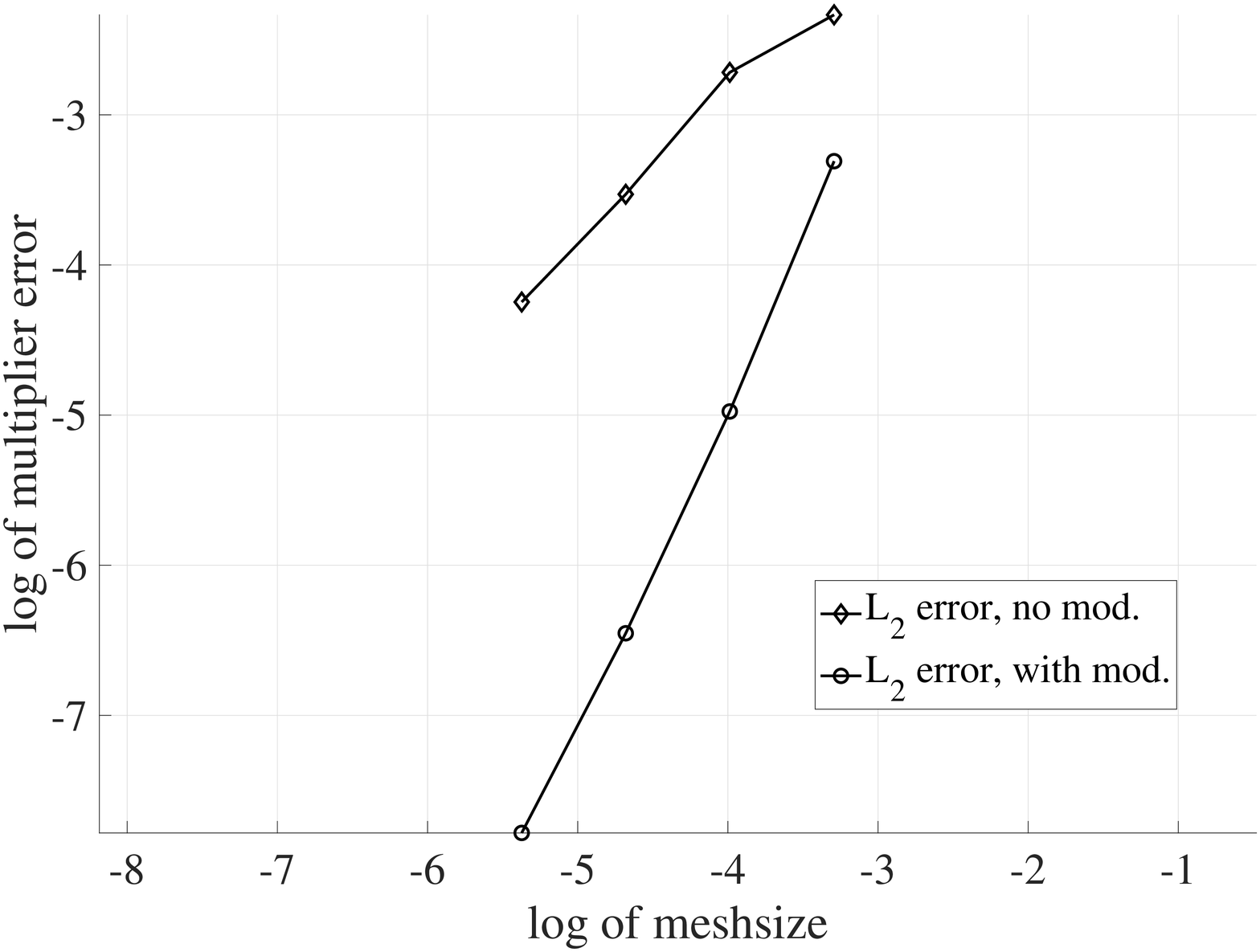}
\end{center}
\caption{Errors in the multiplier with and without boundary modification, $P^2$ case.} \label{fig:errlam}
\end{figure}
\begin{figure}[ht]
\begin{center}\includegraphics[scale=0.3]{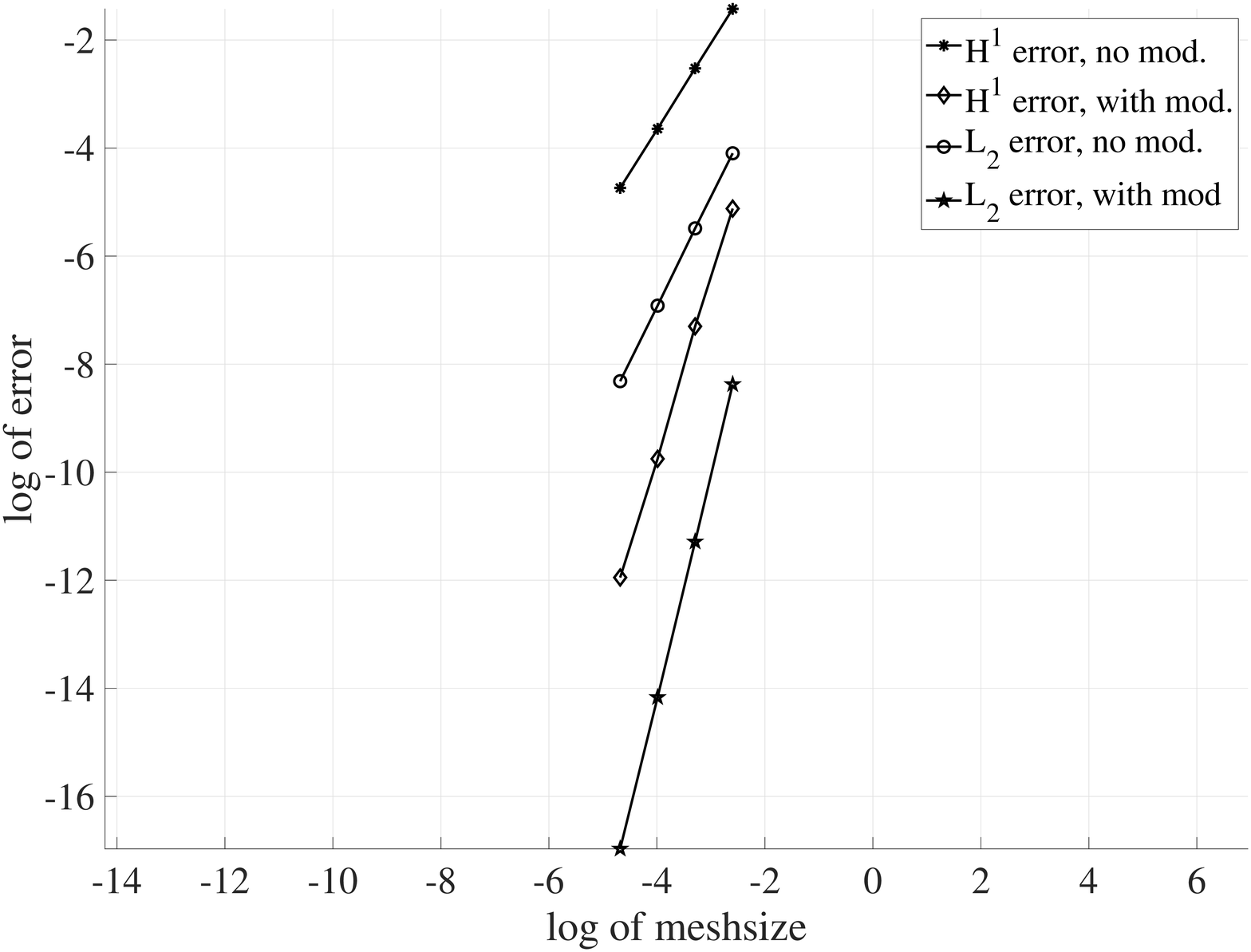}
\end{center}
\caption{Errors with and without boundary modification, $P^3$ case.} \label{fig:errtri2}
\end{figure}
\begin{figure}[ht]
\begin{center}\includegraphics[scale=0.3]{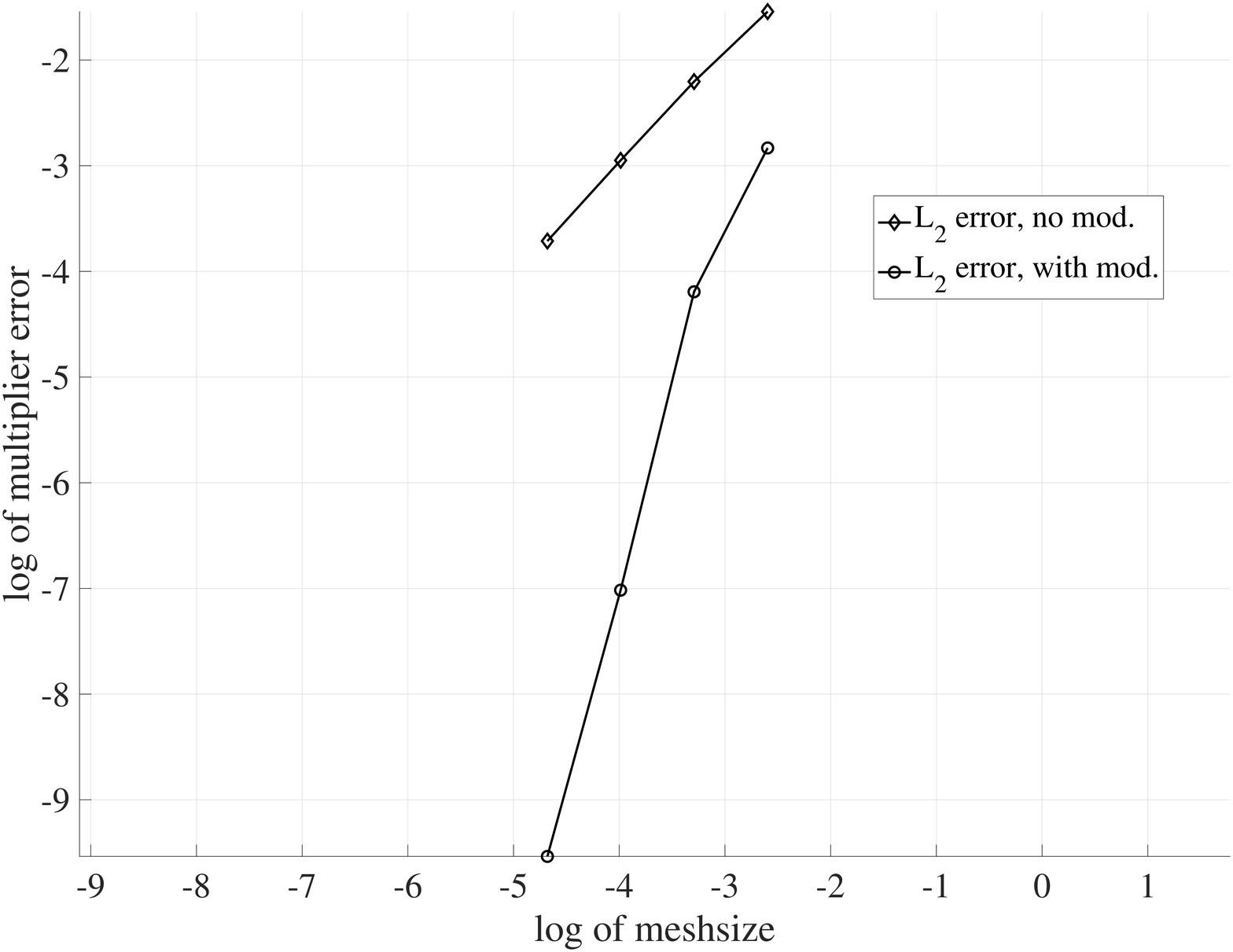}
\end{center}
\caption{Errors in the multiplier with and without boundary modification, $P^3$ case.} \label{fig:errlam2}
\end{figure}
\begin{figure}[ht]
\begin{center}\includegraphics[scale=0.3]{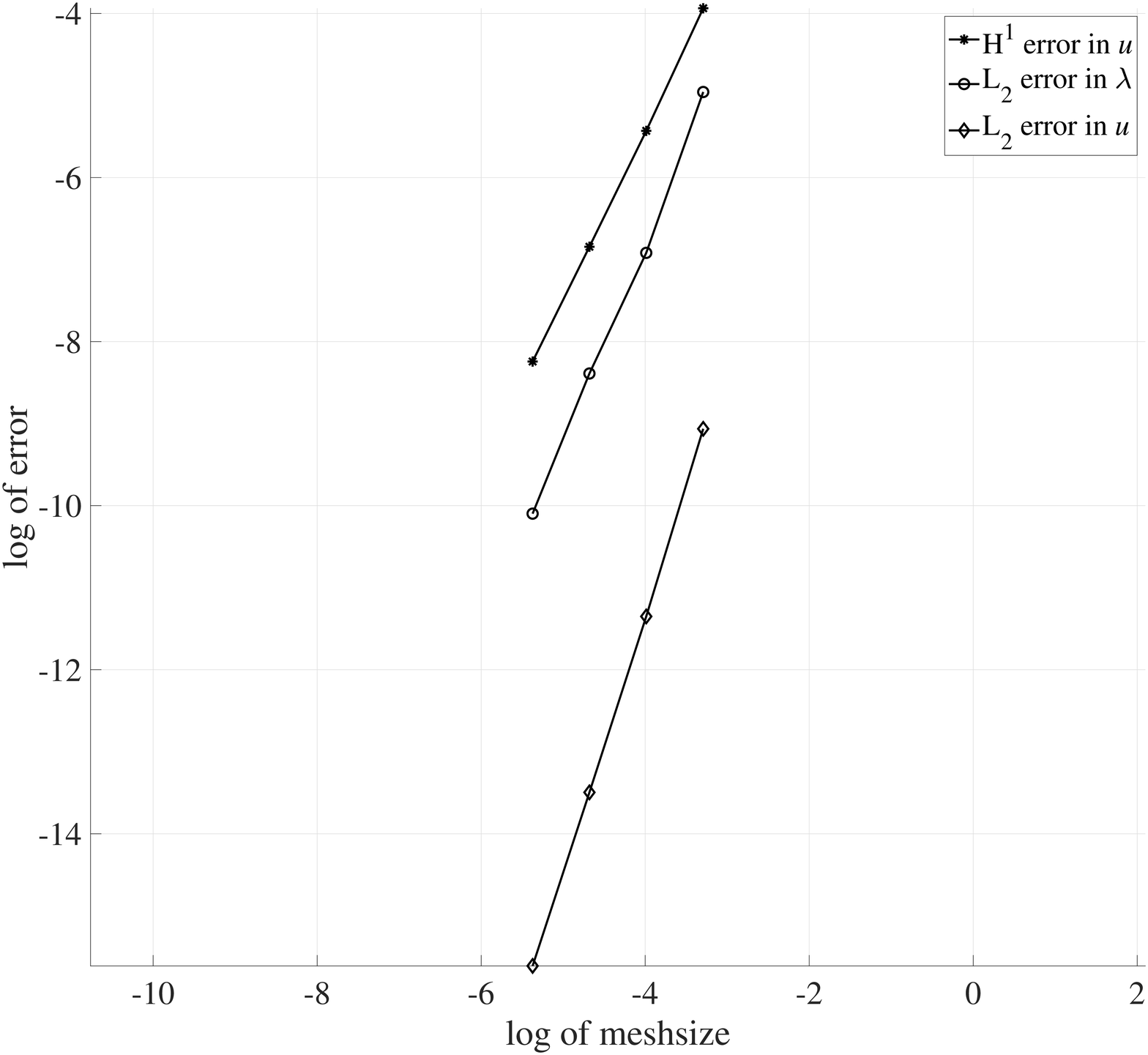}
\end{center}
\caption{Error plots for the unstable triangular element example.} \label{fig:conunstab}
\end{figure}
\begin{figure}[ht]
\begin{center}\includegraphics[scale=0.3]{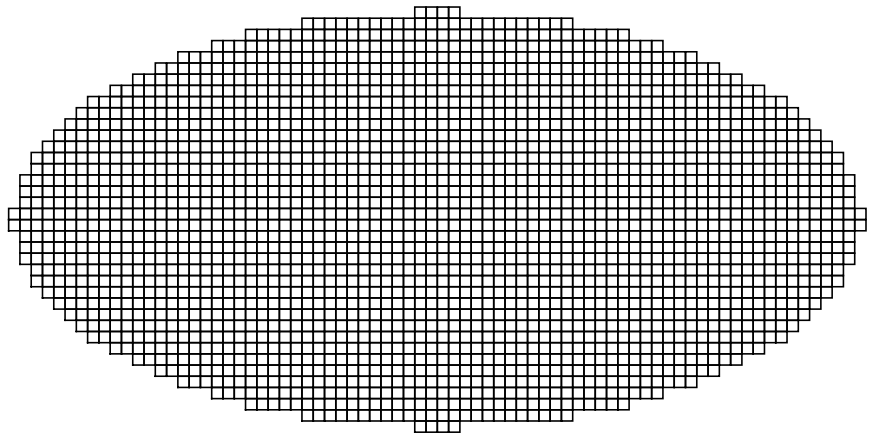}
\end{center}
\caption{A coarse mesh inside the elliptical domain.} \label{fig:coarse}
\end{figure}
\begin{figure}[ht]
\begin{center}\includegraphics[scale=0.24]{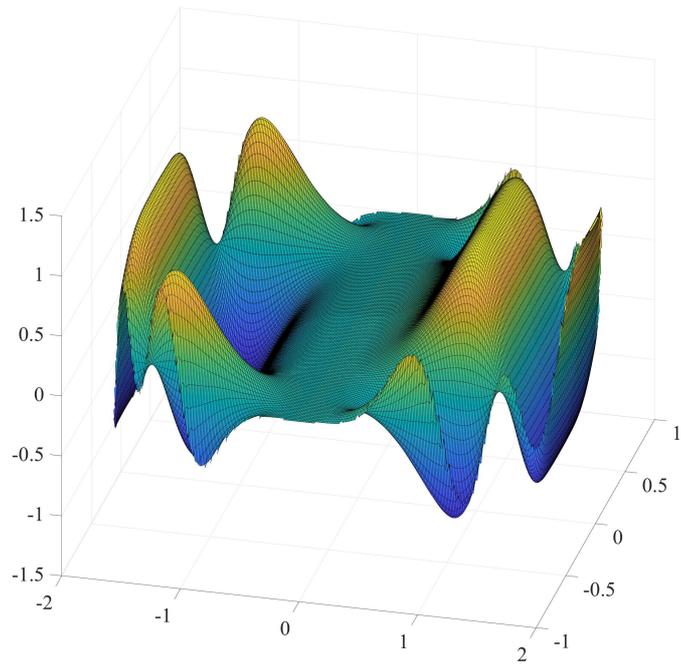}
\end{center}
\begin{center}\includegraphics[scale=0.27]{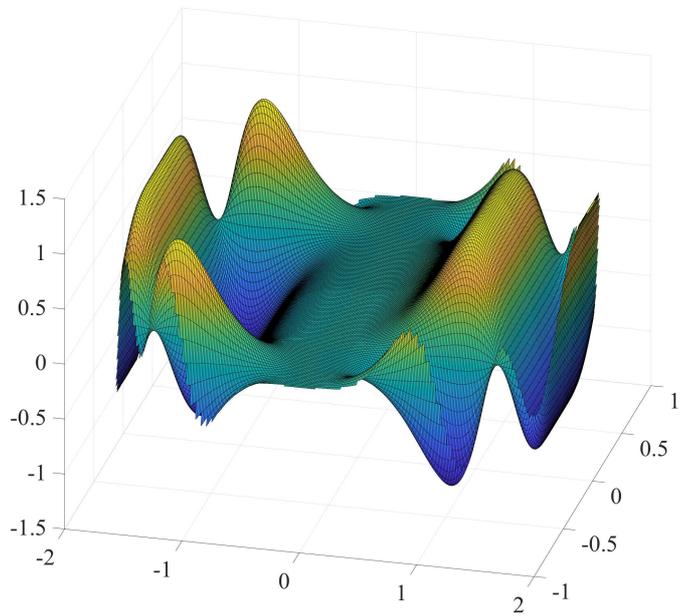}
\end{center}
\caption{Elevation of the discrete solution on rectangles for the unmodified (top) and for the modified (bottom) schemes.} \label{fig:elevq}
\end{figure}
\begin{figure}[ht]
\begin{center}\includegraphics[scale=0.3]{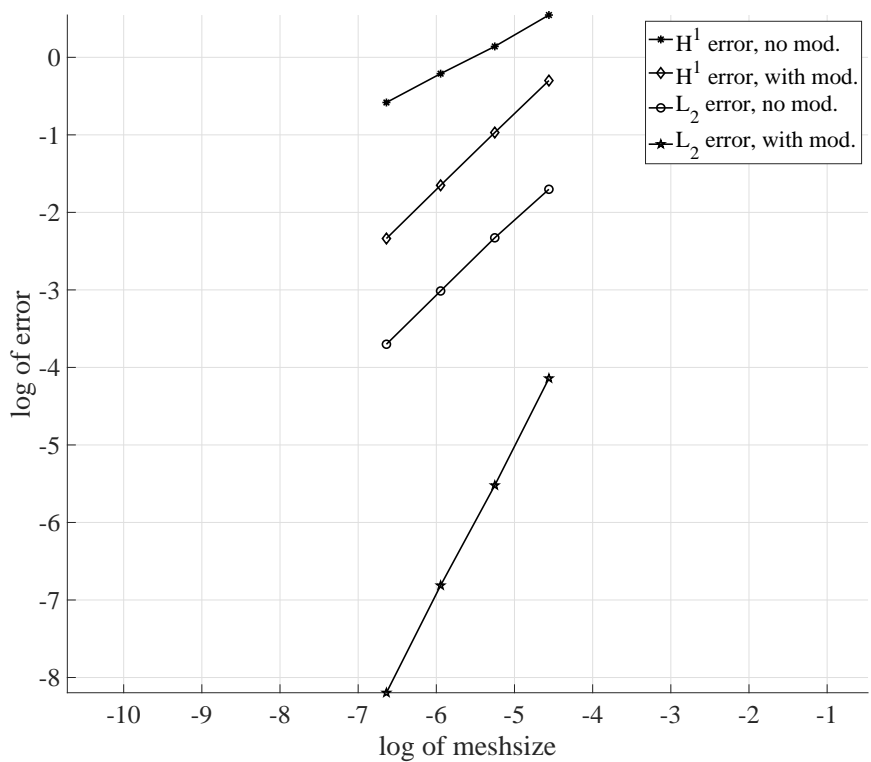}
\end{center}
\caption{Error plots for the rectangular element example.} \label{fig:errquad}
\end{figure}

\end{document}